\RequirePackage[2020-02-02]{latexrelease}
\documentclass[11pt,final,tightenlines,notitlepage,twoside,onecolumn,
nobibnotes,nofootinbib,superscriptaddress]{rcd-nd-eng}

\usepackage{graphicx}
\usepackage{wrapfig}
\usepackage{substr}
\usepackage[tight]{minitoc}
\usepackage{subfigure}
\usepackage{longtable}
\usepackage[all]{xy}
\usepackage{algorithm}
\usepackage{algpseudocode}
\usepackage{array}
\usepackage{float}
\usepackage{placeins}

\renewcommand\proofname{Proof}
\rcdthmsep=\medskipamount
\newrcdthm{teo}{Theorem}
\newrcdthm{lem}{Lemma}
\newrcdthm{pro}{Proposition}
\newrcdthm{assum}{Assumption}
\newrcdthm{cor}{Corollary}
\newrcdthm{fed}{Definition}
\newrcdthm[remark]{rem}{Remark}
\newrcdthm[example]{exl}{Example}
\newrcdthm[example]{com}{Comment}

\newcommand{\R}{\mathbb{R}}
\newcommand{\E}{\mathbb{E}}

\newcommand{\F}{\mathbf F}
\newcommand{\e}{\varepsilon}
\newcommand{\widetO}{\widetilde O}

\newcommand{\eqdef}{:=}
\makeatletter

\AtBeginDocument{}
\def\MakeUppercase#1{#1}

\makeatother

\begin{document}

\def\mtctitle{Contents}%
\def\mtcSfont{\small\rmfamily\upshape\mdseries}   
\makeatletter
\let\mtc@rule=\relax
\makeatother
\allowdisplaybreaks

\widowpenalty =10000
\clubpenalty = 10000

\setcounter{minitocdepth}{2}

\renewcommand{\hm}[1]{#1\nobreak\discretionary{}%
{\hbox{$\mathsurround=0pt #1$}}{}}


\ifx  \Commentx \Undefined
      \long \def \Commentx #1\EndCommentx {}
\else
      \ifx  \EndCommentx \Undefined
            \def \EndCommentx
                {\message {Error: \noexpand \EndCommentx encountered
                                without preceding \noexpand \Commentx}}
      \else 
      \fi
\fi

\Commentx 
This file, "letterspacing.tex", created by Philip Taylor of RHBNC
(P.Taylor@Mail.Rhbnc.Ac.Uk) for Kaveh Bazargan of Focal Image
(Kaveh@Focal.Demon.Co.Uk) may be freely distributed provided that
no changes whatsoever are made to its contents
\EndCommentx

\Commentx 
As far as is possible, `inaccessible' control sequences (i.e.~control sequences
containing one or more commercial-at$\,$s) are used to minimise the risk of
accidental collision with user-defined macros; however, the control sequences
used to access the natural dimensions of the text are required to be user
accessible, and therefore must contain only letters.
\EndCommentx


\newdimen \naturalwidth
\newdimen \naturaldepth
\newdimen \naturalheight

\Commentx 
All control sequences defined hereafter are inaccessible to the casual user,
including the control sequence used to store the category code of commercial-at;
this catcode must be saved in order to be able to re-instate it at end of
module, as we have no idea in what context the module will be |\input|. Having
saved the catcode of commercial-at, we then change it to |11| (i.e.~that of a
letter) in order to allow it to be used in the cs-names which follow.
\EndCommentx


\expandafter \chardef
        \csname \string @code\endcsname =
                \the \catcode `\@
\catcode `\@ = 11

\Commentx 
We will need a a box register in order to measure the natural dimensions of
the text, and a token-list register in which to save the tokens to be
letter-spaced.
\EndCommentx

\newbox \l@tterspacebox
\newtoks \l@tterspacetoks

\Commentx 
We will need to test whether a particular macro expands to a space, so we will
need another macro which does so expand with which to compare it; we will
also need a `more infinite' variant of |\hss|.
\EndCommentx

\def \sp@ce { }
\def \hsss
    {\hskip 0 pt plus 1 fill minus 1 fill\relax}

\Commentx 
Many of the macros will used delimited parameter structures, typically
terminated by the reserved control sequence |\@nd|; we make this a synonym
for the empty macro (which expands to nothing), so that should it accidentally
get expanded there will be no side effects.  We also define a brief synonym
for |\expandafter|, just to keep the individual lines of code reasonably short.
\EndCommentx

\let \@nd = \empty
\let \@x = \expandafter

\Commentx 
We will also need to compare a token which has been peeked at by |\futurelet|
with a space token; because of the difficulty of accessing such a space
token (which would usually be absorbed by a preceding control word), we
establish |\sp@cetoken| as a synonym.  The code to achieve this is messy,
because of the very difficulty just outlined, and so we `waste' a control
sequence |\t@mp|; we then return this to the pool of undefined tokens.
\EndCommentx

\let \sp@cetoken = \relax
\edef \t@mp {\let \sp@cetoken = \sp@ce}
\t@mp \let \t@mp = \undefined

\Commentx 
The user-level macro |\letterspace| has exactly the same syntax as that for
|\hbox| and |\vbox|, as explained in the introduction; the delimited parameter
structure for this macro ensures that everything up to the open brace which
delimits the beginning of the text to be letter-spaced is absorbed as parameter
to the macro, and the brace-delimited text which follows is then assigned to the
token-list register |\l@tterspacetoks|; |\afterassignment| is used to regain
control once the assignment is complete.
\EndCommentx


\def \letterspace #1#%
    {\def \hb@xmodifier {#1}%
     \afterassignment \l@tterspace
     \l@tterspacetoks =
    }

\Commentx 
Control then passes to |\l@tterspace|, which starts by setting an |\hbox|
containing the text to be typeset; the dimensions of this box (and therefore
of the text) are then saved in the open control sequences previously declared,
and the |\hbox| becomes of no further interest.

A new |\hbox| is now created, in which the same text, but this time
letter-spaced, will be set; the box starts and ends with |\hss| glue so that if
only a single character is to be letter-spaced, it will be centered in the box.
If two or more characters are to be letter-spaced, they will be separated by
|\hsss| glue, previously declared, which by virtue of its greater degree of
infinity will completely override the |\hss| glue at the beginning and end;
thus the first and last characters will be set flush with the edges of the box.

The actual mechanism by which letter-spacing takes place is not yet apparent,
but it will be seen that it is crucially dependent on the definition of
|\l@tt@rspace|, which is expanded as the box is being set; the |\@x|
(|\expandafter|) causes the actual tokens stored in |\l@tterspacetoks| to be
made available as parameter to |\l@tt@rspace| rather than the token-list
register itself, as it is these tokens on which |\l@tt@rspace| will operate.
The |\@nd| terminates the parameter list, and the |{}| which immediately
precedes it ensures that that there is always a null element at the end of the
list: without this, the braces which are needed to protect an accent/character
pair would be lost if such a pair formed the final element of the list, at the
point where they are passed as the second (delimited) parameter to |\p@rtiti@n|;
by definition, <TeX> removes the outermost pair of braces from both simple and
delimited parameters during parameter substitution if such braces form the first
and last tokens of the parameter, and thus if a brace-delimited group ever
becomes the second element of a two-element list, the braces will be irrevocably
lost.  The |{}| ensure that such a situation can never occur.
\EndCommentx


\def \l@tterspace
    {\setbox \l@tterspacebox = \hbox
                {\the \l@tterspacetoks}%
     \naturalwidth =  \wd \l@tterspacebox
     \naturaldepth =  \dp \l@tterspacebox
     \naturalheight = \ht \l@tterspacebox
     \hbox \hb@xmodifier
     \bgroup
           \hss
           \@x \l@tt@rspace
           \the \l@tterspacetoks {}\@nd
           \hss
     \egroup
    }

\Commentx 
The next macro is |\l@tt@rspace|, which forms the crux of the entire operation.
The text to be letter-spaced is passed as parameter to this macro, and the
first step is to check whether there is, in fact, any such text; this is
necessary both to cope with pathological usage (e.g.~|\letterspace {}|), and to
provide an exit route, as the macro uses tail-recursion to apply itself
iteratively to the `tail' of the text to be letter-spaced; when no elements
remain, the macro must exit.

Once the presence of text has been ensured, the token-list representing
this text is partitioned into a head (the first element), and the tail
(the remainder);  at each iteration, only the head is considered, although
if the tail is empty (i.e.~the end of the list has been reached), special
action is taken.

If the first element is a space, it is treated specially by surrounding it
with two globs of |\hsss| glue, to provide extra stretchability when compared
to the single glob of |\hsss| glue which will separate consecutive non-space
tokens; otherwise, the element itself is yielded, followed by a single glob
of |\hsss| glue.  This glue is suppressed if the element is the last of the
list, to ensure that the last token aligns with the edge of the box.

When the first element has been dealt with, the macro uses tail
recursion to apply itself to the remaining elements (i.e.~to the tail);
|\@x| (|\expandafter|) is again used to ensure that the tail is expanded
into its component tokens before being passed as parameter.
\EndCommentx


\def \l@tt@rspace #1\@nd
    {\ifx  \@nd #1\@nd
           \let \n@xt = \relax
     \else
           \p@rtition #1\@nd
           \ifx \h@ad \sp@ce \hsss \h@ad \hsss
           \else
                \h@ad
                \ifx \t@il \@nd \else \hsss \fi
           \fi
           \@x \def \@x \n@xt \@x
                {\@x \l@tt@rspace \t@il \@nd}%
     \fi
     \n@xt
    }

\Commentx 
The operation of token-list partitioning is conceptually simple: one passes the
token list as parameter text to a macro defined to take two parameters, the
first simple and the second delimited; the first element of the list will be
split off as parameter-1, and the remaining material passed as parameter-2.
Unfortunately this na\"\i ve approach fails when the first element is a bare
space, as the semantics of <TeX> prevent such a space from ever being passed as
a simple parameter (it could be passed as a delimited parameter, but as one
does not know what token follows the space, defining an appropriate delimiter
structure would be tricky if not impossible).  The technique used here relies
upon the adjunct macro |\m@kespacexplicit|, which replaces a leading bare space
by a space protected by braces; such spaces may legitimately be passed as simple
parameters.  Once that operation has been completed, the re-constructed token
list is passed to |\p@rtiti@n|, which actually performs the partitioning as
above; again |\@x| (|\expandafter|) is used to expand |\b@dy| (the
re-constructed token list) into its component elements before being passed as
parameter text.
\EndCommentx


\def \p@rtition #1\@nd
    {\m@kespacexplicit #1\@nd
     \@x \p@rtiti@n \b@dy \@nd
    }

\def \p@rtiti@n #1#2\@nd
    {\def \h@ad {#1}%
     \def \t@il {#2}%
    }

\Commentx 
The operation of making a space explicit relies on prescience: the code needs
to know what token starts the token list before it knows how to proceed.
Prescience in <TeX> is usually accomplished by |\futurelet|, and this code
is no exception: |\futurelet| here causes |\h@ad| to be |\let| equal to
the leading token, and control is then ceded to |\m@kesp@cexplicit|.

The latter compares |\h@ad| with |\sp@cetoken| (remember the convolutions we
had to go through to get |\sp@cetoken| correctly defined in the first place),
and if they match (i.e.~if the leading token is a space), then |\b@dy|
(the control sequence through which the results of this operation will
be returned) is defined to expand to a protected space (a space surrounded by
braces), followed by the remainder of the elements; if they do not match
(i.e.~if the leading token is \stress {not} a space), then |\b@dy| is simply
defined to be the original token list, unmodified.  If the leading token
\stress {was} a space, it must be replaced by a protected space: this is
accomplished by |\pr@tectspace|.

The |\pr@tectspace| macro uses a delimited parameter structure, as do most
of the other macros in this suite, but the structure used here is unique,
in that the initial delimiter is a space.  Thus, when a token-list starting
with a space is passed as parameter text, that space is regarded as matching
the space in the delimiter structure and removed; the expansion of the macro
is therefore the tokens remaining once the leading space has been removed,
preceded by a protected space |{ }|.
\EndCommentx


\def \m@kespacexplicit #1\@nd
    {\futurelet \h@ad \m@kesp@cexplicit #1\@nd}

\def \m@kesp@cexplicit #1\@nd
    {\ifx \h@ad \sp@cetoken
          \@x \def \@x \b@dy
                     \@x {\pr@tectspace #1\@nd}%
     \else
          \def \b@dy {#1}%
     \fi
    }%

\@x \def \@x \pr@tectspace \sp@ce #1\@nd {{ }#1}

\Commentx 
The final step is to re-instate the category code of commercial-at.
\EndCommentx


\catcode `\@ = \the \@code

\Commentx 
Thus letter-spacing is accomplished.  The author hopes both that the code will
be found useful and that the explanation which accompanies it will be found
informative and interesting.
\EndCommentx

\def\trk#1#2{\mbox{}\letterspace to #1\naturalwidth{#2}}


\renewcommand\refname{References}
\renewcommand{\figurename}{Fig.}


\JournalName{Russian Journal of Nonlinear Dynamics}
\JournalNameEng{Rus. J. Nonlin. Dyn.}
\OrigYearOfIssue{2018}
\OrigVolumeNo{14}
\OrigIssueNo{2}
\NDcite{http://nd.ics.org.ru}


\setcounter{page}{1}

\Rubrika{\relax}
\CRubrika{\relax}
\SubRubrika{\relax}

\beginpaper

\amsmsc{90C25, 90C15, 49M15, 68W15}
\doi{10.20537/nd000000}
\titlerunning{Optimal Inexact Second-Order Acceleration for Distributed Optimization}
\title{Application of Optimal Inexact Second-Order Acceleration to Distributed Stochastic Optimization under Statistical Similarity}
\toctitleeng{Application of Optimal Inexact Second-Order Acceleration to Distributed Stochastic Optimization under Statistical Similarity}

\authorrunning{Y. Sokolov et al.}
\tocauthoreng{Y. A. Sokolov, M. K. Mashtaler, A. V. Gasnikov, M. Tak\'a\v{c}, D. I. Kamzolov, and A. A. Agafonov}

\author{Yury A. Sokolov}
\info{Yury A. Sokolov}{sokolov.ua@mipt.ru}
{Moscow Institute of Physics and Technology}

\author{Maxim K. Mashtaler}
\info{Maxim K. Mashtaler}{mashtaler.mk@phystech.edu}
{Moscow Institute of Physics and Technology}

\author{Alexander V. Gasnikov}
\info{Alexander V. Gasnikov}{gasnikov@yandex.ru}
{Innopolis University\\ Moscow Institute of Physics and Technology}

\author{Martin Tak\'a\v{c}}
\info{Martin Tak\'a\v{c}}{martin.takac@mbzuai.ac.ae}
{Mohamed bin Zayed University of Artificial Intelligence}

\author{Dmitry I. Kamzolov}
\info{Dmitry I. Kamzolov}{kamzolov.opt@gmail.com}
{Independent Researcher}

\author{Artem D. Agafonov}
\info{Artem D. Agafonov}{agafonov.artem98@gmail.com}
{Mohamed bin Zayed University of Artificial Intelligence \\ Moscow Institute of Physics and Technology}

\received 00.00.0000.
\accepted 00.00.0000.
\grants{The research was supported by Russian Science Foundation (project No. 23-11-00229-\foreignlanguage{russian}{П}), https://rscf.ru/en/project/23-11-00229/.}

\begin{abstract}
We consider distributed stochastic convex optimization with a fixed budget of
$N$ independent samples split among $m$ workers. Sample average approximation
reduces the problem to a regularized finite-sum problem whose local Hessians
are statistically similar. This allows the Hessian of the local objective at
the server to be used as an inexact Hessian of the global objective, while the
workers communicate only gradients. We apply the optimal accelerated inexact
Newton extragradient method of~\cite{chen2026optimal} and propose its
distributed restarted variant for the strongly convex empirical problem. The
method reaches the statistical
accuracy of order $N^{-1/2}$ in
$\widetO\left(\max\{N^{1/7},m^{1/4}\}\right)$ communication rounds. Hence,
with $m=N^{4/7}$ workers, it requires $\widetO\left(N^{1/7}\right)$ rounds,
improving the dependence on the total sample size from
$\widetO\left(N^{1/6}\right)$ for the previous accelerated cubic Newton
construction of~\cite{agafonov2021accelerated}. Each iteration uses two
gradient aggregation rounds and does not require Hessian communication.
\end{abstract}

\maketitle\label{article_begin}
{
\newpage
\section{Introduction}

In this work, we consider the distributed stochastic optimization problem
\begin{equation}
\label{eq:stoch_problem}
\min \limits_{x\in\R^d} \lbrace \F(x) \eqdef \E_{\xi}f(x,\xi) \rbrace,
\end{equation}
where $\xi$ is a random variable, e.g. random data, and $f$ is convex and sufficiently smooth, which implies that $\F$ is convex.  Let $x^*$ be a
solution of~\eqref{eq:stoch_problem}. We assume access to $m$ workers and a
total fixed budget of $N$ independent realizations of $\xi$. Let $T$ denote
the number of communication rounds. Under the fixed sample budget $N$, we
are interested in reaching accuracy $\E \F(x_T)-\F(x^*)\leq\e$
using as few communication rounds $T$ and as many workers $m$ as possible.

In this work we use two smoothness assumptions. For all $\xi$, $f(x, \xi)$ is $L_0$-Lipschitz continuous 
\begin{equation}
    \label{eq:lip_func}
    |f(x, \xi) - f(y, \xi)| \leq L_0 \|x - y\|, \quad \forall x, y \in \R^d,
\end{equation}
and has $L$-Lipschitz continuous Hessian
\begin{equation*}
    \|\nabla^2 f(x, \xi) - \nabla^2 f(y, \xi)\| \leq L \|x - y\|, \quad \forall x, y \in \R^d.
\end{equation*}
Two standard approaches to problem~\eqref{eq:stoch_problem} are stochastic
approximation (SA)~\cite{robbins1951stochastic,polyak1990new, polyak1992acceleration,nemirovski2009robust,lan2012optimal} and sample average approximation (SAA)~\cite{shapiro2014lectures,dvinskikh2020stochastic}.

Agafonov et al.~\cite{agafonov2021accelerated} showed that second-order
information can reduce the communication complexity of SAA-based distributed
methods. Their analysis follows three steps. First, SAA reduces the stochastic
problem to a regularized strongly convex finite-sum problem. Second, the iid
partition of the sample yields statistical similarity between local and global
empirical Hessians. Third, the local Hessian at the master node is used as an
inexact Hessian of the global objective. An accelerated cubic Newton
method~\cite{ghadimi2017second,agafonov2024advancing} is then applied with
restarts. We use the same reduction and replace the accelerated cubic Newton
method with the recently proposed accelerated inexact Newton extragradient
(AINE) method~\cite{chen2026optimal}, which is optimal for convex optimization
with inexact Hessians. We next introduce the SAA reduction and statistical
similarity. Section~\ref{sec:distributed_aine} then formalizes the
inexact-Hessian model and gives the distributed AINE method and its restart
scheme.

\paragraph{Sample average approximation.} In this paper, we follow the SAA approach~\cite{agafonov2021accelerated} by sampling $N$ iid realizations $\xi^{k, j}$ in advance and constructing
regularized finite-sum problem
\begin{equation}
    \label{eq:finite_sum_problem}
    \min \limits_{x\in\R^d} \left \lbrace F(x) \eqdef
    \frac{1}{m}\sum_{k=1}^{m}f_k(x)
    +\frac{\mu}{2}\|x-x_0\|^2 \right \rbrace,
\end{equation}
where $f_k(x) = \tfrac{1}{n} \sum_{j=1}^n f(x, \xi^{k, j})$. The main idea of SAA approach is that $x^*$ can be approximated by the solution of~\eqref{eq:finite_sum_problem} $x_F^*$. We assume that $x^*$ and $x_F^*$ belong to the Euclidean ball centered at $x_0$ with radius $R$.

Following~\cite{agafonov2021accelerated}, by~\cite[Corollary 1.2]{feldman2019high} under the $L_0$-Lipschitz continuity assumption~\eqref{eq:lip_func} and an appropriate choice of the
regularization parameter $\mu =\frac{L_0 \log N}{R\sqrt N} $, with probability at least $1-\delta$,
$
\F(x_F^*)-\F(x^*)
\leq
O\left(
\frac{L_0R}{\sqrt N}
\log\frac{N}{\delta}
\right)
$.
Since $\F$ is $L_0$-Lipschitz continuous,
$
\F(x)-\F(x_F^*)
\leq
L_0\|x-x_F^*\|.
$
Thus, for $x=x_T$,
we obtain
\begin{equation}
\label{eq:saa_transfer}
\F(x_T)-\F(x^*)
\leq
L_0\|x_T-x_F^*\|
+
O\left(
\frac{L_0R}{\sqrt N}
\log\frac{N}{\delta}
\right).
\end{equation}
Thus, it is sufficient to find a good approximation $x_T$ to the solution
$x_F^*$ of the finite-sum problem~\eqref{eq:finite_sum_problem}.



\paragraph{Statistical similarity.} The finite-sum representation provides one more important property that can be exploited by distributed second-order methods. Under the assumption that $\xi^{k,j}$ are iid, the following bound holds for all $k$ with probability at least $1-\delta$:
\begin{equation}
\label{eq:statistical_similarity}
\sup_w\left\|
\frac{1}{m}\sum_{j=1}^{m}\nabla^2 f_j(w)-\nabla^2 f_k(w)
\right\|
\leq \beta =
\widetO\left(\sqrt{\frac{L L_0 d}{n}}\right).
\end{equation}

\paragraph{Distributed acceleration with inexact Hessians.}
We follow the approach of~\cite{agafonov2021accelerated} and use local Hessian at the server as an approximation of global Hessian. We propose the distributed version of the Accelerated Inexact Newton Extragradient method of~\cite{chen2026optimal} with optimal convergence rate $O\left(
\frac{\beta R^2}{T^2}
+
\frac{LR^3}{T^{7/2}} \right)$ in convex case.
For strongly convex case, applying the restart technique with the SAA transfer~\eqref{eq:saa_transfer} and similarity bound~\eqref{eq:inexact_hessian} we obtain
\begin{equation}
\label{eq:aine_population_rate}
\F(x_T)-\F(x^*)
\leq
\widetO\left(
\exp\left(
-\min\left\{
\frac{T}{N^{1/7}},
\frac{T}{m^{1/4}}
\right\}
\right)
+
\frac{1}{\sqrt N}
\right).
\end{equation}
The comparison with~\cite{agafonov2021accelerated} is summarized in
Table~\ref{tab:main_comparison}. The values of $T$ and $m$ in the table are
obtained by balancing the optimization and statistical errors and then using
as many workers as possible without increasing the communication complexity.

\begin{table}[H]
\caption{Comparison with the accelerated cubic Newton method of
\cite{agafonov2021accelerated}. }
\label{tab:main_comparison}
\centering
\small
\renewcommand{\arraystretch}{1.2}
\setlength{\tabcolsep}{4pt}
\begin{tabular}{|p{0.3\textwidth}|>{\centering\arraybackslash}p{0.4\textwidth}|c|c|}
\hline
Method & Bound & $T$ & $m$ \\
\hline
Distributed Accelerated cubic Newton~\cite{agafonov2021accelerated}
& $\widetO\left(
\exp\left(
-\min\left\{
\tfrac{T}{N^{1/6}},
\tfrac{T}{m^{1/4}}
\right\}
\right)
+\tfrac{1}{\sqrt N}
\right)$
& $N^{1/6}$
& $N^{2/3}$ \\
\hline
Distributed AINE [this work]
& $\widetO\left(
\exp\left(
-\min\left\{
\tfrac{T}{N^{1/7}},
\tfrac{T}{m^{1/4}}
\right\}
\right)
+\tfrac{1}{\sqrt N}
\right)$
& $N^{1/7}$
& $N^{4/7}$ \\
\hline
\end{tabular}
\end{table}

The new method improves the communication dependence on the sample size from
$N^{1/6}$ to $N^{1/7}$. The balanced number of workers changes from
$N^{2/3}$ to $N^{4/7}$, which makes the communication--parallelization trade-off
explicit.

\subsection{Related works}

\paragraph{Distributed optimization under statistical similarity.}
Statistical similarity, also referred to as second-order similarity or Hessian
similarity, has been extensively used to reduce the communication complexity of
distributed empirical risk minimization. DANE~\cite{shamir2014communication}
uses local subproblems corrected by the global gradient.
DiSCO~\cite{zhang2015disco} implements an inexact damped Newton method using
distributed preconditioned conjugate gradients.
The statistically preconditioned accelerated gradient method
of~\cite{hendrikx2020statistically} uses a local empirical objective as a
preconditioner for the global problem.

The closest work to ours is~\cite{agafonov2021accelerated}. The authors use the
local Hessian at the master node as an inexact Hessian of the global objective
and apply an accelerated cubic Newton method. They further combine the
finite-sum convergence rate with SAA and obtain a guarantee for the original
stochastic optimization problem. We use the same statistical reduction,
communication model, and local-Hessian construction. The difference is that
we replace their accelerated cubic Newton method with AINE.

Several communication-efficient methods under similarity have appeared after
\cite{agafonov2021accelerated}. Kovalev et al.~\cite{kovalev2022optimal}
proposed an accelerated gradient-sliding method that matches lower complexity
bounds on both communication rounds and local gradient computations in the
full-participation strongly convex setting. Khaled and
Jin~\cite{khaled2022faster} combined second-order similarity with client
sampling, approximate proximal steps, and variance reduction.
Lin et al.~\cite{lin2023stochastic} studied average second-order similarity,
proposed the SVRS and AccSVRS methods, and established nearly matching lower
bounds. Jiang et al.~\cite{jiang2024federated} revisited DANE and proposed
DANE+ and FedRed, which allow inexact local solvers and improve local
computational complexity. Their subsequent stabilized proximal-point
methods~\cite{jiang2024stabilized} support partial client participation,
stochastic local solvers, acceleration, and adaptive line search.

These works consider broader communication and local computation trade-offs
for finite-sum or federated optimization. Our goal is different. We study the
rate obtained by inserting an optimal inexact second-order method into the
second-order SAA construction of~\cite{agafonov2021accelerated}. Therefore,
Table~\ref{tab:main_comparison} compares methods within this construction.

\paragraph{Distributed and inexact second-order methods.}
Another line of work develops distributed Newton-type methods that learn or
compress approximations of global curvature. FedNL~\cite{safaryan2022fednl} learns
compressed local Hessian approximations and provides variants with partial
participation, cubic regularization, and line search.
FLECS~\cite{agafonov2022flecs} combines low-dimensional Hessian sketches with
compression to reduce the memory and communication costs of Hessian learning.
In contrast, our method neither learns nor communicates an approximation of
the global Hessian. It uses the local Hessian available at the master node,
while the workers communicate gradients only.

For centralized convex optimization, Ghadimi
et al.~\cite{ghadimi2017second} introduced accelerated cubic regularization
with inexact Hessians. Inexactness was subsequently studied for arbitrary-order
tensor methods in~\cite{agafonov2023inexact}. Globally convergent accelerated
second-order methods with stochastic gradient and Hessian oracles were
developed in~\cite{antonakopoulos2022extra,agafonov2024advancing}.
Chen et al.~\cite{chen2026optimal} proposed AINE and established the optimal
inexact second-order oracle complexity given by~\eqref{eq:aine-complexity} for
convex objectives with Lipschitz continuous Hessians. We use AINE without
modifying its acceleration rule. Our analysis concerns its distributed
implementation under statistical similarity, its restart for strongly convex
finite-sum objectives, and the subsequent SAA transfer.

\subsection{Our contribution}

The main contribution of this paper is twofold.

\paragraph{Distributed finite-sum optimization.}
We apply AINE to the strongly convex finite-sum
problem~\eqref{eq:finite_sum_problem}. Statistical
similarity~\eqref{eq:statistical_similarity} makes the Hessian available at the
master node a $\beta$-inexact Hessian~\eqref{eq:inexact_hessian}, so the AINE
oracle can be implemented by~\eqref{eq:crn-oracle}. We give a distributed
implementation that uses two gradient aggregations per iteration, keeps the
local Hessian at the master node, and attains the restarted communication
complexity~\eqref{eq:restart-complexity}.

\paragraph{Distributed stochastic optimization.}
We combine the finite-sum result with the SAA
transfer~\eqref{eq:saa_transfer} and obtain the population
guarantee~\eqref{eq:aine_population_rate}. The method reaches the statistical
accuracy in the number of communication rounds given
by~\eqref{eq:round-complexity}. The largest number of workers that preserves
this complexity is given by~\eqref{eq:parallelization}. Within the second-order
SAA construction of~\cite{agafonov2021accelerated}, this improves the
communication dependence on the total sample size from $N^{1/6}$ to
$N^{1/7}$.

\section{Distributed AINE under Statistical Similarity}
\label{sec:distributed_aine}

We call a symmetric matrix $H(x)$ a $\beta$-inexact Hessian of $F$ at $x$ if
\begin{equation}
\label{eq:inexact_hessian}
\|H(x)-\nabla^2F(x)\|\leq\beta.
\end{equation}
For the SAA problem~\eqref{eq:finite_sum_problem}, we use
$H(x)=\nabla^2f_1(x)+\mu I$. The common quadratic regularizer cancels in the
Hessian difference. Therefore, the statistical similarity
bound~\eqref{eq:statistical_similarity} implies~\eqref{eq:inexact_hessian}.

We apply AINE~\cite[Algorithm~2]{chen2026optimal}. The algorithm below
specializes it to the second-order case $p=2$ with Hessian inexactness $\beta$.

A $\beta$-inexact Monteiro--Svaiter oracle queried at $z$ returns $(y,\lambda)$
such that
\begin{equation}\label{eq:ms-oracle}
 \lVert\nabla F(y)+\lambda(y-z)\rVert
 \leq\frac\lambda2\lVert y-z\rVert,\qquad
 \lambda\geq c\bigl(\beta+L\lVert y-z\rVert\bigr).
\end{equation}
By~\cite[Lemma~3.1]{chen2026optimal}, it is implemented with the inexact Hessian
$H(z)$ by the cubic-regularized Newton subproblem
\begin{align}\label{eq:crn-oracle}
 \widetilde F(u,z)={}&F(z)+\langle\nabla F(z),u-z\rangle\notag\\
 &+\frac12\langle(H(z)+2\beta I)(u-z),u-z\rangle
 +\frac L3\lVert u-z\rVert^3,\notag\\
 y={}&\arg\min_u\widetilde F(u,z),\notag\\
 \lambda={}&2\beta+L\lVert y-z\rVert.
\end{align}

We use the master/workers architecture of
\cite[Section~3]{agafonov2021accelerated}.  Without loss of generality, agent
$1$ is chosen as the central node (server), and the remaining $m-1$ agents are
workers connected to the server.  Each agent $k$ stores its local objective
$f_k$.  At an oracle query point $z$, the server broadcasts $z$ to the workers.
Every worker computes $\nabla f_k(z)$ and sends this vector to the server, which
forms
\[
 \nabla F(z)=\frac1m\sum_{k=1}^m\nabla f_k(z)+\mu(z-x_0).
\]
The server computes the local Hessian $\nabla^2f_1(z)$, constructs the model
$\widetilde F(\cdot,z)$ in~\eqref{eq:crn-oracle}, and obtains $y$ by minimizing
this model.  It then broadcasts $y$ to the workers.  The workers return
$\nabla f_k(y)$, allowing the server to form $\nabla F(y)$ and complete the
AINE update.

Thus, one AINE iteration requires two communication rounds: one aggregation at
$z$ and one at $y$.  Each communication transmits only vectors, while the
Hessian $\nabla^2f_1(z)$ is computed once per iteration and remains on the
server.  The local sample size on the server may also be increased; this
reduces the inexactness parameter $\beta$ in
~\eqref{eq:statistical_similarity} without changing the communicated objects.

\begin{algorithm}
\caption{AINE specialized to $p=2$ with Hessian inexactness $\beta$
\cite[Algorithm~2]{chen2026optimal}}\label{alg:aine}
\begin{algorithmic}[1]
\Require $x_0$, horizon $t$, the oracle \eqref{eq:ms-oracle}, and a known
bound $D_0\geq\lVert x_0-x_F^*\rVert^2/2$
\State $v_0=z_0=x_0$, $A_0=0$
\State $(y_0,\lambda_0)\gets\mathcal O_2^\beta(z_0)$
\For{$j=0,\ldots,t-1$}
  \State If $j=0$, set $\lambda'_j=\lambda_0$
  \State Elif $j=1$ or $A_j>2A_s$, store $s=j$, $A_s=A_j$, and set
  \Statex \hspace{2em}
  $\displaystyle\lambda'_j=\Theta\left(
  \beta+\left(\frac{D_0L^2}{A_s^{3/2}}\right)^{2/7}\right)$
  \State Otherwise set $\lambda'_j=\lambda'_s$
  \State $a'_{j+1}\gets
  (1+\sqrt{1+4\lambda'_jA_j})/(2\lambda'_j)$,
  $A'_{j+1}\gets A_j+a'_{j+1}$
  \State $z_j\gets(A_jx_j+a'_{j+1}v_j)/A'_{j+1}$
  \If{$j>0$}
    \State $(y_j,\lambda_j)\gets\mathcal O_2^\beta(z_j)$
  \EndIf
  \State $\theta_j\gets\min\{1,\lambda'_j/\lambda_j\}$,
  $a_{j+1}\gets\theta_ja'_{j+1}$,
  $A_{j+1}\gets A_j+a_{j+1}$
  \State $x_{j+1}\gets
  ((1-\theta_j)A_jx_j+\theta_jA'_{j+1}y_j)/A_{j+1}$
  \State $v_{j+1}\gets v_j-a_{j+1}\nabla F(y_j)$
\EndFor
\State \Return $x_t$
\end{algorithmic}
\end{algorithm}

The parameter $D_0$ need not be the exact initial distance: any known upper
bound on $\lVert x_0-x_F^*\rVert^2/2$ is sufficient.  In particular, a
certified radius $R_0\geq\lVert x_0-x_F^*\rVert$ allows the choice
$D_0=R_0^2/2$.  Theorem~3.1 of~\cite{chen2026optimal}, specialized to $p=2$
with Hessian inexactness $\beta$, yields
\begin{equation}\label{eq:aine-rate}
 F(x_t)-F(x_F^*)
 =O\left(
 \frac{\beta D_0}{t^2}
 +\frac{L D_0^{3/2}}{t^{7/2}}
 \right).
\end{equation}
Equivalently, the number of AINE iterations required to reach objective
accuracy $\varepsilon$ is
\begin{equation}\label{eq:aine-complexity}
 O\left(
 \sqrt{\frac{\beta R_0^2}{\varepsilon}}
 +\left(\frac{LR_0^3}{\varepsilon}\right)^{2/7}
 \right),\qquad R_0\geq\lVert x_0-x_F^*\rVert.
\end{equation}

\subsection{Restarted method for strongly convex objectives}

The empirical objective $F$ in~\eqref{eq:finite_sum_problem} is $\mu$-strongly
convex.  We apply the restart construction of
\cite[Section~3]{agafonov2021accelerated} to the convex rate \eqref{eq:aine-rate}.

\begin{algorithm}
\caption{Restarted distributed AINE}\label{alg:restart}
\begin{algorithmic}[1]
\Require $z_0$, $R_0\geq\lVert z_0-x_F^*\rVert$, and $\mu>0$
\For{$s=1,2,\ldots$}
  \State $R_{s-1}\gets R_0/2^{s-1}$
  \State $x_0^{(s)}\gets z_{s-1}$, $v_0^{(s)}\gets z_{s-1}$,
  $z_0^{(s)}\gets z_{s-1}$, $A_0^{(s)}\gets0$
  \State $D_{0,s}\gets R_{s-1}^2/2$
  \State Choose
  \begin{equation}\label{eq:stage-length}
   t_s=\left\lceil C\max\left\{
   \sqrt{\frac\beta\mu},
   \left(\frac{LR_{s-1}}\mu\right)^{2/7}
   \right\}\right\rceil
  \end{equation}
  for a sufficiently large universal constant $C$
  \State Run Algorithm~\ref{alg:aine} for $t_s$ iterations with the initialized
  state $(x_0^{(s)},v_0^{(s)},z_0^{(s)},A_0^{(s)})$ and bound $D_{0,s}$
  \State $z_s\gets x_{t_s}^{(s)}$
\EndFor
\end{algorithmic}
\end{algorithm}

\begin{cor}\label{cor:restart-rate}
The iterates of Algorithm~\ref{alg:restart} satisfy
\begin{equation}\label{eq:restart-stage-rate}
 \lVert z_s-x_F^*\rVert\leq R_0\,2^{-s},\qquad
 F(z_s)-F(x_F^*)\leq\frac{\mu R_0^2}{2}\,2^{-2s}.
\end{equation}
Counting two communication rounds per iteration, set
\begin{equation}\label{eq:restart-scales}
 \tau_1=1+C\left(\frac{LR_0}\mu\right)^{2/7},\qquad
 \tau_2=1+C\sqrt{\frac\beta\mu}.
\end{equation}
For an arbitrary communication budget $T$, let
\begin{equation}\label{eq:completed-stages}
 S(T)=\left\lfloor\frac{T}{2\max\{\tau_1,\tau_2\}}\right\rfloor,
 \qquad x_T:=z_{S(T)}.
\end{equation}
Thus $x_T$ is the last completed restart point if the budget expires inside a
stage.  The iterates satisfy
\begin{equation}\label{eq:communication-rate}
 \lVert x_T-x_F^*\rVert
 \leq 2R_0\exp\left(-\frac{(\ln 2)T}
 {2\max\{\tau_1,\tau_2\}}\right).
\end{equation}
Moreover, reaching $F(z_s)-F(x_F^*)\leq\varepsilon$ requires
\begin{equation}\label{eq:restart-complexity}
 T=O\left(
  \left(\frac{LR_0}\mu\right)^{2/7}
  +\left(1+\sqrt{\frac\beta\mu}\right)
  \log\frac{\mu R_0^2}{\varepsilon}
 \right)
\end{equation}
communication rounds.
\end{cor}

\begin{proof}
If $\lVert z_{s-1}-x_F^*\rVert\leq R_{s-1}$, then the initialization in
Algorithm~\ref{alg:restart} gives
$D_{0,s}=R_{s-1}^2/2\geq\lVert z_{s-1}-x_F^*\rVert^2/2$.
Hence Theorem~3.1 of~\cite{chen2026optimal}, equivalently
\eqref{eq:aine-rate}, applies at stage $s$.  Together with
\eqref{eq:stage-length}, it gives
$F(z_s)-F(x_F^*)\leq\mu R_{s-1}^2/8$.  Strong convexity consequently gives
$\lVert z_s-x_F^*\rVert\leq R_{s-1}/2=R_s$, and induction proves
\eqref{eq:restart-stage-rate}.  Since $R_{s-1}\leq R_0$ and
$\lceil a\rceil\leq a+1$, the definitions in
\eqref{eq:stage-length}--\eqref{eq:restart-scales} imply
$t_s\leq\max\{\tau_1,\tau_2\}$.  Each stage uses two communication rounds per
iteration, so at least $S(T)$ stages are completed within the budget $T$.
Combining this fact with \eqref{eq:restart-stage-rate} and
$2^{-\lfloor a\rfloor}\leq2\exp(-a\ln2)$ proves
\eqref{eq:communication-rate}.  Finally, summing
the geometrically decreasing $(LR_{s-1}/\mu)^{2/7}$ terms and the constant
$1+\sqrt{\beta/\mu}$ term over the restart stages gives
\eqref{eq:restart-complexity}.
\end{proof}

\section{Application to Stochastic Optimization Problem}
\label{sec:4}

Combining the SAA transfer~\eqref{eq:saa_transfer}
with~\eqref{eq:communication-rate} gives
\begin{equation}\label{eq:population-rate}
 \boldsymbol F(x_T)-\boldsymbol F(x^*)
 \leq\widetO\left(
 2L_0R_0\exp\left(-\frac{(\ln 2)T}
 {2\max\{\tau_1,\tau_2\}}\right)
 +\frac{L_0R}{\sqrt N}
 \right).
\end{equation}
Thus, the population error is the sum of the optimization error and the
finite-sample approximation error.

Using $\mu=\widetO(N^{-1/2})$ and
$\beta=\widetO((m/N)^{1/2})$ from~\eqref{eq:statistical_similarity}
in~\eqref{eq:restart-scales} yields
\begin{equation}\label{eq:distributed-scales}
 \tau_1\simeq N^{1/7},\qquad \tau_2\simeq m^{1/4}.
\end{equation}
Thus the optimization and statistical terms in \eqref{eq:population-rate}
have the same order after
\begin{equation}\label{eq:round-complexity}
 T=\widetO\left(\max\{N^{1/7},m^{1/4}\}\right)
\end{equation}
communication rounds.  Full parallelization means
increasing $m$ only until it starts increasing the required number of rounds.
Balancing the two terms in \eqref{eq:distributed-scales} gives
\begin{equation}\label{eq:parallelization}
 m=N^{4/7},\qquad T=\widetO(N^{1/7}).
\end{equation}
For comparison, the accelerated cubic method in
\cite[Section~4]{agafonov2021accelerated} gives $m=N^{2/3}$ and
$T=\widetO(N^{1/6})$.


\section{Numerical Experiments}
\label{sec:experiments}

We evaluate the construction of Section~\ref{sec:distributed_aine} on logistic
regression.  The experiments
compare the server and global Hessians, report convergence against both outer
iterations and communicated bits, and examine how the server sample size affects
the accelerated method.

\subsection{Setup}

We use the LIBSVM~\cite{chang2011libsvm} \texttt{a9a} training set: rows normalized to unit Euclidean
norm and one deterministic tail sample dropped, leaving $N=32{,}560$ examples of
dimension $d=123$ split evenly over $m$ nodes, so each node holds $n=N/m$ of them
and node~$1$ is the server.  Each realization is a labelled example
$\xi=(a,b)\in\R^d\times\{-1,+1\}$, and $f$ in~\eqref{eq:finite_sum_problem} is
the logistic loss
\begin{equation}
    \label{eq:logistic_loss}
    f(x,\xi)=\log\bigl(1+\exp(-b\,a^\top x)\bigr).
\end{equation}
We take $x_0=\mathbf 1$ and $\mu=0$.  Unit row norms give the certified upper
bound $L \leq 1/(6\sqrt3)$ for the Hessian-Lipschitz constant.

The two partitions differ in how the index set is cut.  The \emph{IID} split
assigns each node a uniformly random slice of the examples, which is the
sampling model behind~\eqref{eq:statistical_similarity}.  The
\emph{heterogeneous} split instead draws a direction $u$ uniformly from the unit
sphere, orders the examples by the projections $\langle a_i,u\rangle$, and cuts
the ordered list into $m$ consecutive blocks of equal size.  Writing
$i_1,\ldots,i_N$ for the resulting order,
\begin{equation}
    \label{eq:feature_split}
    \langle a_{i_1},u\rangle\leq\cdots\leq\langle a_{i_N},u\rangle,
    \qquad
    \mathcal I_k=\{\,i_{(k-1)n+1},\ldots,i_{kn}\,\},
    \quad k=1,\ldots,m .
\end{equation}
Different nodes then hold different regions of feature space and the examples on
a node are not an independent sample of the population, so the reasoning behind
\eqref{eq:statistical_similarity} no longer applies.  Figures~\ref{fig:exp-iid}
and~\ref{fig:exp-hetero} fix $m=16$, and Figure~\ref{fig:exp-setups} varies the
block.  Figures~\ref{fig:exp-iid}, \ref{fig:exp-hetero},
and~\ref{fig:exp-setups} report medians over three seeds, while
Figure~\ref{fig:exp-similarity} reports medians and interquartile ranges over
five seeds.

All methods aggregate the gradient of the full objective and differ only in the
Hessian.  Cubic regularized Newton~\cite{nesterov2006cubic}, its
Nesterov-accelerated form~\cite{nesterov2008accelerating}, and
Algorithm~\ref{alg:aine} are each run with the exact $\nabla^2F$ and with the
server $\nabla^2f_1$; the exact-Hessian runs are controls, since assembling
$\nabla^2F$ requires every node to transmit a $d\times d$ block.  Only the uplink
is charged.  The communication axis reports payload per worker: a vector costs
$32d$ bits and a symmetric Hessian $32d(d+1)/2$ bits, so one Hessian is worth
$(d+1)/2$ gradients.  
The accelerated methods spend two gradient aggregations per outer
iteration against one for the unaccelerated method.

\paragraph{Constants.}
Two quantities have to be fixed before the runs, the cubic regularization
constant and the inexactness parameter, and they are fixed in different ways.

The cubic constant is tuned per method.  For each method we sweep a grid of
  values jointly with the free parameter of its accelerated scheme.  For ACRN
  this is the weight of the extrapolation step, the coefficient with which the
  previous iterate enters the point at which the model is built, and for
  Algorithm~\ref{alg:aine} it is the scale of the guess $\lambda'$ in the
  Guess subroutine.  The plots therefore compare the best stable configuration of
  each method rather than all methods under one common cubic constant.

The inexactness parameter is swept on the same grid.  The condition
     Algorithm~\ref{alg:aine} verifies at every iteration, the first
     in~\eqref{eq:ms-oracle}, involves only $\nabla F(y)$, $\lambda$ and $y-z$, never
     $\nabla^2F$.  A run whose $\beta$ is too small to satisfy it stops there, and we
     report the best of the runs that complete.  ACRN performs no such check, so its
     configuration is selected on the final gap alone.      
\subsection{Convergence}

With the exact Hessian, Algorithm~\ref{alg:aine} attains the smallest final
objective gap in Figures~\ref{fig:exp-iid} and~\ref{fig:exp-hetero}~ Its curve
is close to that of the Nesterov-accelerated cubic method over much of the run,
and the separation becomes visible only at the longer horizon shown here.

Under the IID split, Algorithm~\ref{alg:aine} driven by $\nabla^2f_1$ reaches
the smallest final gap among the three server-Hessian variants.  Its normalized
communication curve is also below the other server-Hessian accelerated curve
at the end of the reported run.  Under the heterogeneous split it is second
among the server-Hessian accelerated methods, consistently with the larger
Hessian mismatch measured for this split in Figure~\ref{fig:exp-similarity}.


\begin{figure}[h]
\makebox[\textwidth][l]{%
\includegraphics{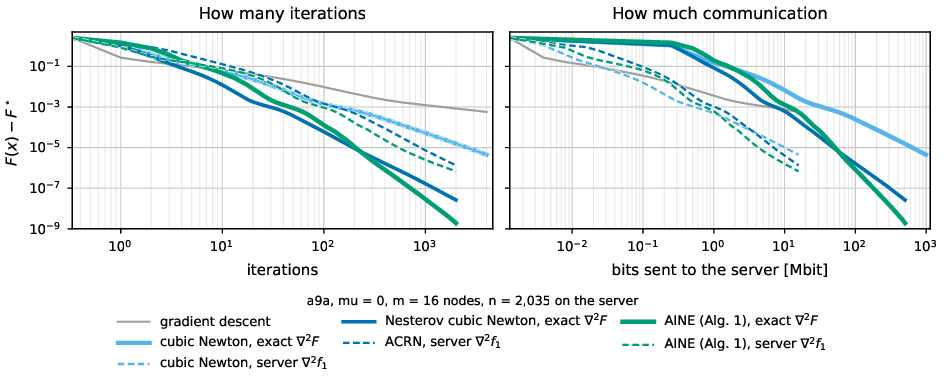}}
\caption{IID split.}
\label{fig:exp-iid}
\end{figure}


\begin{figure}[h]
\makebox[\textwidth][l]{%
\includegraphics{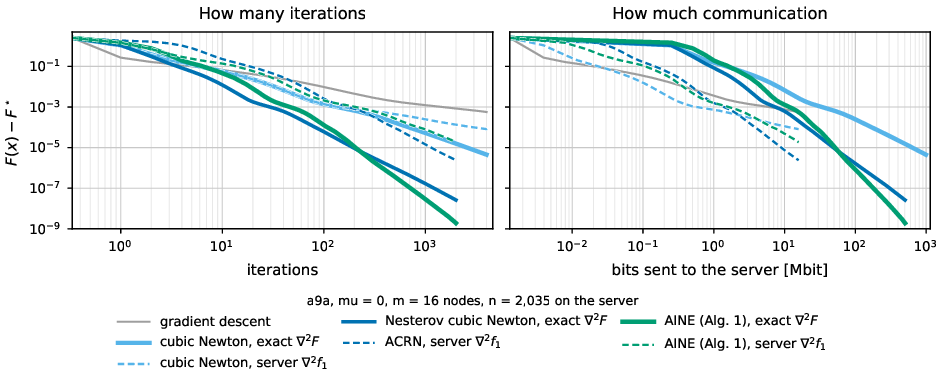}}
\caption{Heterogeneous split.}
\label{fig:exp-hetero}
\end{figure}

\FloatBarrier
\subsection{The mismatch and the certificate}

Figure~\ref{fig:exp-similarity} measures
$\lVert\nabla^2f_1(x_t)-\nabla^2F(x_t)\rVert$ along the run. 
Between the two splits at the same block size it differs by
more than an order of magnitude, which is what separates the two figures above.

The covariance certificate shown in Figure~\ref{fig:exp-similarity}~ is nearly flat in $n$ and
substantially exceeds the realized mismatch on this problem.  Supplying such a
conservative value increases the term $2\beta I$ in~\eqref{eq:crn-oracle} and
reduces the practical benefit of the server Hessian.  Thus, constructing a
computable certificate that follows the realized similarity remains an important
practical issue.

\begin{figure}[!htbp]
\centering
\includegraphics[width=0.80\textwidth]{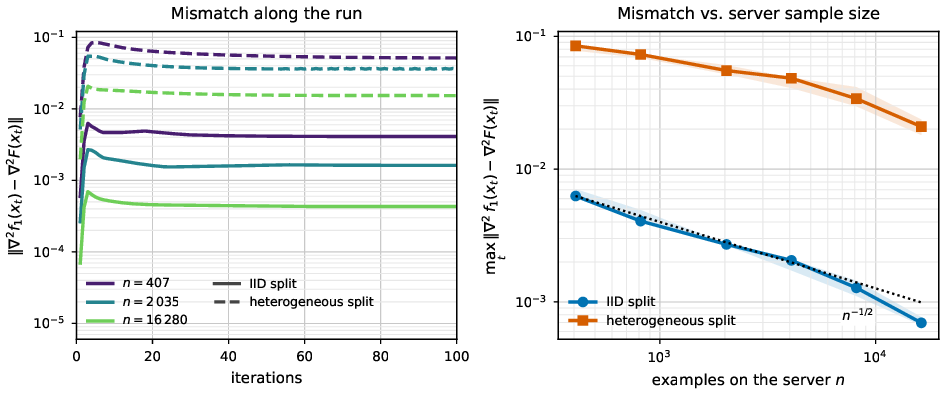}
\caption{Hessian mismatch.}
\label{fig:exp-similarity}
\end{figure}

\FloatBarrier
\subsection{Server block size}

Figure~\ref{fig:exp-setups} isolates the effect of the server block.  Under the
IID split, the final gap of the server-Hessian AINE run decreases as the server
block grows over the range considered.  The corresponding changes for ACRN \cite{agafonov2021accelerated}
and for the uncorrected substitution are smaller at this horizon.  The three
methods use the same full gradients and differ in how they incorporate the
server Hessian.  Under the heterogeneous split, increasing the block gives a
much smaller improvement, consistently with the slower decay of the measured
mismatch in Figure~\ref{fig:exp-similarity}.

\begin{figure}[!htbp]
\centering
\includegraphics[width=0.85\textwidth]{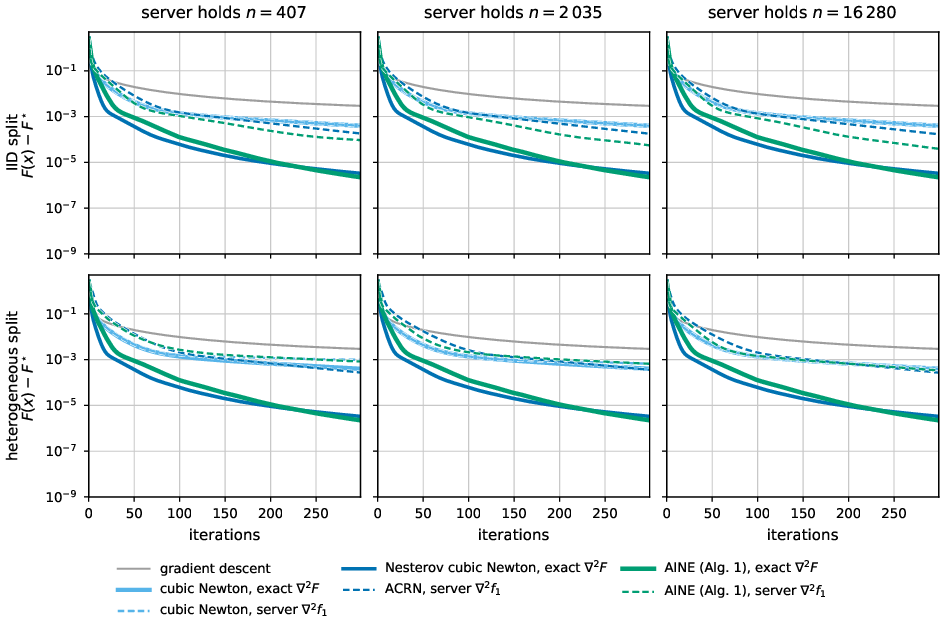}
\caption{Split against server block size, $\mu=0$.~}
\label{fig:exp-setups}
\end{figure}

\FloatBarrier

\section{Conclusion}

In this paper, we applied optimal acceleration with inexact Hessians to
distributed stochastic optimization under statistical similarity. We showed
that the Hessian of the local empirical objective stored at the server can be
used as an inexact Hessian of the global empirical objective. This gives a
distributed implementation of AINE in which the workers communicate only
gradients. Combined with restarts and sample average approximation, the method
reaches the statistical accuracy in
$\widetO\left(\max\{N^{1/7},m^{1/4}\}\right)$ communication rounds. Choosing
$m=N^{4/7}$ gives the complexity $\widetO\left(N^{1/7}\right)$ and improves
the dependence on the sample size compared with the previous accelerated
cubic Newton construction.

The experiments support the role of statistical similarity in this
construction. Under an IID partition, increasing the server sample size reduces
the measured local-to-global Hessian mismatch and improves the final accuracy of
the server-Hessian AINE run. Under the heterogeneous partition, the mismatch is
larger and decays more slowly, and the same gain is substantially weaker. The
results therefore agree with the dependence on statistical similarity predicted
by the theory.

\section*{Acknowledgments}

The research was supported by Russian Science Foundation (project
No.~23-11-00229), \url{https://rscf.ru/en/project/23-11-00229/}.

\appendix

\bibliographystyle{unsrtnat}
\bibliography{agafonov}
\label{LastBibItem:}






}
\label{article_end}

\endpaper
\end{document}